\documentclass[12pt]{amsart}

\usepackage{amsmath}
\usepackage{amssymb}  
\usepackage{latexsym} 
\usepackage{comment}
\usepackage{url}
\usepackage{fullpage,url,amssymb,amsmath,amsthm,amsfonts,mathrsfs}
\usepackage[usenames,dvipsnames]{color}
\usepackage[pagebackref = true, colorlinks = true, linkcolor = blue, citecolor = Green]{hyperref}
\usepackage[alphabetic,lite]{amsrefs}
\usepackage{enumitem}
\usepackage{amscd}   
\usepackage[all, cmtip]{xy} 
\usepackage{xfrac}
\usepackage[T1]{fontenc}

\DeclareFontEncoding{OT2}{}{} 
\newcommand{\textcyr}[1]{%
 {\fontencoding{OT2}\fontfamily{wncyr}\fontseries{m}\fontshape{n}\selectfont #1}}

\usepackage[all]{xy}
\usepackage{fullpage}
\newcommand{\Sha}{{\mbox{\textcyr{Sh}}}}

\usepackage{color} 

\def\act#1#2%
  {\mathop{}%
   \mathopen{\vphantom{#2}}^{#1}%
   \kern-3\scriptspace%
   #2}

\newcommand{\Z}{{\mathbb Z}}

\newcommand{\Q}{{\mathbb Q}}

\newcommand{\F}{{\mathbb F}}

\newcommand{\Kbar}{{\overline{K}}}

\newcommand{\kbar}{{\overline{k}}}

\newcommand{\calO}{{\mathcal O}}

\DeclareMathOperator{\Sel}{Sel}

\DeclareMathOperator{\End}{End}
\DeclareMathOperator{\Hom}{Hom}

\DeclareMathOperator{\rk}{rk}
\DeclareMathOperator{\Tr}{Tr}

\DeclareMathOperator{\Gal}{Gal}

\DeclareMathOperator{\HH}{H}

\DeclareMathOperator{\Aut}{Aut}

\DeclareMathOperator{\Mor}{Mor}

\newtheorem{Theorem}{Theorem}[section]
\newtheorem{Lemma}[Theorem]{Lemma}
\newtheorem{Proposition}[Theorem]{Proposition}
\newtheorem{Corollary}[Theorem]{Corollary}
\newtheorem{Definition}[Theorem]{Definition}

\newtheorem{Remark}[Theorem]{Remark}

\numberwithin{equation}{section}

\begin{document}
\title{Tate-Shafarevich groups of constant elliptic curves and isogeny volcanos}
\author{Brendan Creutz}
\address{School of Mathematics and Statistics, University of Canterbury, Private Bag 4800, Christchurch 8140, New Zealand}
\email{brendan.creutz@canterbury.ac.nz}
\urladdr{http://www.math.canterbury.ac.nz/\~{}b.creutz}

\author{Jos\'e Felipe Voloch}
\address{School of Mathematics and Statistics, University of Canterbury, Private Bag 4800, Christchurch 8140, New Zealand}
\email{felipe.voloch@canterbury.ac.nz}
\urladdr{http://www.math.canterbury.ac.nz/\~{}f.voloch}

\begin{abstract}
We describe the structure of Tate-Shafarevich groups of a constant elliptic curves over function fields by 
exploiting the volcano structure of isogeny graphs of elliptic curves over finite fields.

\end{abstract}

\maketitle

\section{Introduction}

Let $k$ be a finite field and $E/k$ an elliptic curve. Let $F/k$ 
be another curve and set $K = k(F)$ its function field. The base change $E/K$ is usually referred to
as a constant elliptic curve over the global function field $K$. Already since the 60's, from the work of Tate completed
by Milne \cite{Milne}, it has been known that the full Birch and Swinnerton-Dyer conjecture holds for such constant elliptic
curves and, in particular, that the Tate-Shafarevich group $\Sha(E/K)$ is finite. But not much information on
this group beyond what the Birch and Swinnerton-Dyer conjecture gave was obtained at the time. Later, $\Sha(E/K)$ was computed for supersingular $E$ and various fields $K$ by Elkies,
Shioda and others (\cite{Shioda} and references therein) motivated by the fact that the Mordell-Weil group $E(K)$ with its canonical
height pairing gave interesting examples of lattice packings. In those examples, $\Sha(E/K)$ is a $p$-group,
where $p$ is the characteristic of $k$.

Parallel to these developments, motivated by applications to computational number theory and cryptography,
there have been a number of papers studying the $\ell$-isogeny graphs on elliptic curves (\cite{Drew} and references
therein).
These graphs are defined by taking as vertices elliptic curves over a field $k$ and edges representing isogenies
of degree $\ell$ between them. A fundamental insight of Kohel (Theorem \ref{kohel} below) was extended and reformulated
by Fouquet and Morain \cite{FM} to the statement that a conencted component of the $\ell$-isogeny graph for a prime $\ell$
over a finite field $k$ consisting of ordinary curves has a specific structure referred to as $\ell$-volcano graph. We will provide
the precise description below.

The purpose of this paper is to bring these two trends together and use the structure of the isogeny graphs to
obtain explicit description of the Tate-Shafarevich group $\Sha(E/K)$ in many instances and, indeed, go further
and compute all terms of the descent exact sequence for $\ell$-isogenies, when $F$ is also an elliptic curve.

\section{Preliminaries}

\subsection{Isogeny graphs}

\begin{Definition}
A graph with vertex set $V$ is called an $\ell$-volcano graph if its vertex set can be partitioned as 
$V = V_1 \cup V_2  \cdots \cup V_m$, where $m$ is called the height of the volcano, $V_1$ the base and
$V_m$ the crater or top. In addition, the induced graph on $V_m$ is a cycle (the edges on this subgraph are
called horizontal), the degree of all vertices
not on $V_1$ is $\ell + 1$ , the degree of the vertices in $V_1$ is $1$, for each vertex on $V_i, i <m$, there
is a unique edge from it to a vertex in $V_{i+1}$ (these are called upward edges) and for each vertex on $V_i, i > 1$, the other edges go to
vertices in $V_{i-1}$ (these are called downward edges). The vertices in $V_i$ are said to have height (or level) $i$.
\end{Definition}

\begin{Remark}
\label{trees}
Note that, if we remove the edges of $V$ connecting the vertices of $V_m$,
the connected components of the resulting graph are trees with roots at the 
the vertices of $V_m$. We call those the trees of the volcano.
\end{Remark}

We will assume throughout that $\ell$ is prime and,
not the characteristic of $k$ (the case $\ell = p$ will be briefly mentioned in a comment). 
As stated above, a connected component of the 
$\ell$-isogeny graph of ordinary curves over the finite field $k$ has the structure of a $\ell$-volcano (with a minor
modification if the component contains a curve with $j$-invariant $0$ or $1728$).
Moreover, denoting by $h_\ell(E)$ the height (or level) of $E$ in its connected component of
the $\ell$-isogeny graph, viewed as a volcano, we have $h_\ell(E) = v_\ell([\End(E):\Z[\pi]])$. We refer to an isogeny of degree a power of $\ell$ as upward, downward or horizontal if the corresponding path on the $\ell$-isogeny graph is a sequence of edges with the same property.
	

\begin{Theorem}\cite[Propositions 21 and 22]{Kohel}\label{kohel}
Let $E$ and $E'$ be isogenous ordinary elliptic curves over $k$. Then $h_\ell(E) = h_\ell(E')$ if and only if there is an isogeny $E \to E'$ of degree prime to $\ell$.
\end{Theorem}

\subsection{Descent}
 
Briefly, let us recall the descent mechanism and the definition of the Selmer and Tate-Shafarevich groups. For
details, see \cite[Chapter X]{Silverman}.

Given an isogeny $\phi: E \to E_1$ defined over a field $K$, we have the exact sequence
$$0 \to \ker \phi \to E \to E_1 \to 0$$
which gives, by taking cohomology, the Kummer sequence,

$$0 \to E_1(K)/\phi(E(K)) \to \HH^1(K,\ker \phi) \to \HH^1(K,E)[\phi] \to 0.$$

If $K$ is a global field, the corresponding exact sequences for each completion $K_v$ of $K$ can also be considered.
By comparing the corresponding sequences, one obtains the descent sequence

\begin{equation}
\label{des-seq}
0 \to E_1(K)/\phi(E(K)) \to \Sel^\phi(E_1/K) \to  \Sha(E/K)[\phi] \to 0 
\end{equation}

\noindent
where $\Sel^\phi(E_1/K)$ is the Selmer group consisting of the elements of $\HH^1(K,\ker \phi)$ whose image in
$\prod_v \HH^1(K_v,\ker \phi)$ is also on the image of $\prod_v E_1(K_v)$ (where the direct product is
taken over all places $v$ of $K$) and 

$$\Sha(E/K) := \ker \left(\HH^1(K,E) \to \bigoplus_v \HH^1(K_v,E)\right)$$

\noindent
is the Tate-Shafarevich group.

 \subsection{Constant elliptic curves}
 
 If $E$ is an elliptic curve over a finite field $k$ and $F$ is another curve defined over $k$ with function field $K = k(F)$, then a basic fact, central to our approach, is the identification $E(K) = \Mor(F,E)$. If, in addition $F$ is also an elliptic curve with identity $0_F$ then, 
 evaluating a map $\alpha \in E(K) = \Mor(F,E)$ at $0_F$ gives a retraction of the inclusion $E(k) \subset E(K)$. This gives a a split exact sequence,
 
\begin{equation}
 \label{split}
0 \to E(k) \to E(K) \to \Hom(F,E) \to 0\,. 
\end{equation}

The following lemma is central to our determination of the first term in~\eqref{des-seq}.

 \begin{Lemma}\label{lem1}
	Suppose $\phi:E \to E_1$ is a an isogeny of $\ell$-primary degree. If $\phi$ is downward, then $\Hom(E,E_1) = \phi \End(E)$. If $\phi$ has prime degree, then $\Hom(E,E_1) = \phi \End(E)$ if and only if $\phi$ is downward or $E \simeq E_1$ and $\phi$ is self dual up to an automorphism.
\end{Lemma}

\begin{proof}
	First suppose $\phi$ is downward and $\alpha \in \Hom(E,E_1)$. We can factor $\alpha$ as $\alpha = \alpha_\ell\circ \alpha'$ with $\alpha' : E \to E'$ an isogeny of degree prime to $\ell$ and $\alpha_\ell:E' \to E_1$ an isogeny of $\ell$-primary degree. By Theorem~\ref{kohel}, $h_\ell(E) = h_\ell(E')$ which is greater than $h_\ell(E_1)$ by assumption. From the structure of the $\ell$-sogeny graph (see Remark~\ref{trees}) the final edge in any path from $E'$ to $E_1$ in the $\ell$-isogeny graph must be $\phi$. Therefore $\alpha_\ell \in \phi\Hom(E',E_1)$ and so $\alpha \in \phi \End(E)$.
	
	For the remainder of the proof we suppose $\phi$ has degree $\ell$. If $\phi$ is not downward, then it is either upward or horizontal. We consider each case in turn.
	
	First suppose $\phi$ is upward and let $L = \End(E_1) \otimes \Q$. Then $\phi$ determines an injective ring homomorphism $\iota: \End(E) \to L$ by $\gamma \mapsto \phi\gamma\hat\phi \otimes \ell^{-1}$. This induces an inclusion $\End(E) \subset \End(E_1) \subset L$ with $[\End(E_1):\End(E)] = \ell$ since $\phi$ is upward. Applying the case already proven to the downward isogeny $\hat\phi$ we get $\Hom(E_1,E) = \hat\phi \End(E_1)$. Dualizing gives $\Hom(E,E_1) = \End(E_1)\phi$. If we also had $\Hom(E,E_1) = \phi \End(E)$, then multiplying by $\hat\phi$ on the right would give $\phi\End(E)\hat{\phi} = \End(E_1)\ell$. But, this implies $\iota(\End(E)) = \End(E_1)$, which is a contradiction.
	
	Now consider the case when $\phi$ is horizontal. This can only happen when $\End(E) = \calO_L$ is the maximal order in $L = \End(E) \otimes \Q$ and $\ell = \frak{p}\frak{q}$ is not inert. Moreover, $\phi = \phi_\frak{p}$ or $\phi= \phi_\frak{q}$, where for an ideal $\frak{a} \subset \End(E)$, the notation here means that $\ker(\phi_\frak{a}) = \{ P \in E \;:\; \alpha(P) = 0 \text{, for all $\alpha \in \frak{a}$}\}$. Relabeling if necessary we can suppose $\phi = \phi_\frak{p}$. If $\frak{p}$ is principal, then $E_1 \simeq E$ and $\hat\phi \in \phi \End(E)$ is equivalent to $\hat{\phi} \in \phi\Aut(E)$. Now suppose these ideals are not principal. Then every endomorphism of $E$ of degree divisible by $\ell$ is itself divisible by $[\ell]$. Let $m$ be the order of $\frak{p}$ in the class group of $\mathcal{O}_L$. Then $\phi_{\frak{q}^{m-1}} \in \Hom(E,E_1)$ corresponds to the path along the crater rim opposite from $\phi = \phi_\frak{p}$. We claim that $\phi_{\frak{q}^{m-1}} \notin \phi\End(E)$. Indeed, if $\phi_{\frak{q}^{m-1}} = \phi_\frak{p}\gamma$ with $\gamma \in \End(E)$, then being an endomorphism of $\ell$-primary degree $\gamma = [\ell^r]$ for some $r$. But this implies $\frak{q}^m = \frak{q}^r\frak{p}^{r+1}$, which contradicts unique factorization of ideals in $\calO_L$.
\end{proof}

\section{Calculation of the Selmer group}

	Suppose $\phi : E \to E_1$ is a separable isogeny of degree $n$ over the finite field $k$ with $n$ relatively prime to the characteristic of $k$. Let $F/k$ be an elliptic curve and set $K = k(F)$. Set $G_k = \Gal(\kbar/k)$.
	
	\begin{Proposition}\label{prop:selunramified}
		The $\phi$-Selmer group, $\Sel^\phi(E'/K)$, is the unramified subgroup, $$\ker\left(\HH^1(K,E[\phi]) \to \prod_v\HH^1(\kbar K_v,E[\phi])\right)\,.$$
	\end{Proposition}
	
	\begin{proof}
		It is enough to show that for any place $v$ of $K$ the image of $E_1(K_v)/\phi (E(K_v))$ in $\HH^1(K_v,E[\phi])$ under the map in~\eqref{des-seq} coincides with the kernel of the restriction map $\HH^1(K_v,E[\phi]) \to \HH^1(\kbar K_v,E[\phi])$. Suppose $k_v$ is the residue field of $K_v$. 	The inflation-restriction sequence yields a commutative diagram with exact row
		\[
			\xymatrix{
			& E_1(k_v)/\phi(E(k_v)) \ar[r] \ar[d] & E_1(K_v)/\phi(E(K_v)) \ar[d]\\
			0 \ar[r]& \HH^1(k_v,E[\phi]) \ar[r] & \HH^1(K_v,E[\phi]) \ar[r]& \HH^1(\kbar K_v,E[\phi]) \,.}
		\]
		We claim that the arrows eminating from the top left are both isomorphisms, from which the result easily follows.

		The top horizontal map is induced by the inclusion $E_1(k_v) \subset E_1(K_v)$. It is an isomorphism since the kernel of the reduction map $E_1(K_v) \to E_1(k_v)$ is $\ell$-divisible for every $\ell$ not divisible by $p = \operatorname{char}(k)$ \cite[Propositions IV.2.3 and VI.2.2]{Silverman} and $\deg(\phi)$ is prime to $p$. That the vertical arrow is an isomorphism follows from exactness of~\eqref{des-seq} and the fact that $\HH^1(k,E) = 0$ since every smooth genus $1$ curve over $k_v$ has a rational point.
	\end{proof}
	
	Define $r_\phi : \Hom(F,E') \to \Hom(F[n],E[\phi])$ by $r_\phi(\alpha) = \left.\phi^*\alpha\right|_{F[n]}$, where the notation means the restriction of the isogeny $\phi^*\alpha$ to $F[n]$.
	
	\begin{Lemma}
		$r_\phi(\alpha) = (\left.\alpha^*\right|_{E'[\phi^*]})^D$, the Cartier dual of the morphism of Galois modules obtained by restricting $\alpha^*$ to kernel of $\phi^*$.
	\end{Lemma}
	
	\begin{proof}
		The Weil pairing gives isomorphisms $F[n] \ni x \mapsto e_n(x,\bullet) \in \Hom(F[n],\mu_n) = F[n]^D$ and $E[\phi] \ni x \mapsto e_\phi(x,\bullet) \in E'[\phi^*]^D$. The Cartier dual of $\alpha^*$ restricted to $F[n]$ is the map $F[n]^D \ni e_n(x,\bullet) \mapsto e_n(x,\alpha^*(\bullet))$. For $y \in E'[\phi^*]$ functoriality of the Weil pairing gives that
	\[
			e_n^F(x,\alpha^*(y)) = e_n^{E'}(\alpha(x),y) = e_\phi(\phi^*\alpha(x),y).
	\]
	Therefore $(\left.\alpha^*\right|_{E'[\phi^*]})^D$ is the map sending $x$ to $\phi^*\alpha(x) = r_\phi(\alpha)(x)$.
	\end{proof}
	
	\begin{Lemma}
		$r_\phi(\alpha) = 0$ if and only if $\alpha \in \phi \Hom(F,E)$.
	\end{Lemma}
	
	\begin{proof}
		If $\alpha = \phi\alpha'$, then $r_\phi(\alpha) = \phi^*\phi\alpha' = n\alpha' = 0$ when restricted to $F[n]$. Conversely, suppose $\phi^*\alpha = 0$ on $F[n]$. Then the isogeny $\phi^*\alpha$ must factor through multiplication by $n$. This means there is some $\beta \in \Hom(F,E')$ such that $\phi^*\alpha = n\beta = \phi^*\phi\beta$. Therefore the map $\alpha - \phi \beta$ takes values in $E'[\phi^*]$. But a difference of isogenies is either surjective or the zero map. Hence $\alpha = \phi\beta$.
	\end{proof}
	
	\begin{Remark}\label{rem:phi=l}
	In the case that $\phi = [\ell]$ is multiplication by $\ell$ on $E$, $r_\phi(\alpha) = \ell\alpha \in \Hom(F[\ell^2],E[\ell])$. This can be identified with $\left.\alpha\right|_{F[\ell]} \in \Hom(F[\ell],E[\ell])$ as follows. Every element of the Galois module $\Hom(F[\ell^2],E[\ell])$ is divisible by $\ell$. The map sending $\beta = \ell\beta' \in \Hom(F[\ell^2],E[\ell])$ to $\left.\beta'\right|_{F[\ell]} \in \Hom(F[\ell],E[\ell])$ is an isomorphism sending $r_\ell(\alpha)$ to $\left.\alpha\right|_{F[\ell]}$.
	\end{Remark}

	\begin{Theorem}
		There is an exact and commutative diagram
		\[
			\xymatrix{
			0 \ar[r]& E'(k)/\phi(E(k)) \ar[d] \ar[r]& E'(k)/\phi(E(k)) \ar[d]^{\partial_\phi} \\
			0 \ar[r]& E'(K)/\phi(E(K)) \ar[d] \ar[r]^{\partial_\phi} & \Sel^\phi(E'/K) \ar[r]\ar[d] &  \Sha(E/K)[\phi] \ar[r] \ar@{=}[d] & 0 \\
			0 \ar[r]& \frac{\Hom_{G_k}(F,E')}{\phi\Hom_{G_k}(F,E)} \ar[r]^{r_\phi}& \Hom_{G_k}(F[n],E[{\phi]}) \ar[r] \ar[r] & \Sha(E/K)[\phi] \ar[r] & 0
			}
		\]
	\end{Theorem}
	
	\begin{proof}
		Let $G = \Gal(K^{\operatorname{sep}}/\kbar K)$. Since $E[\phi]^G = E[\phi]$, the inflation-restriction sequence gives
		\[
			0 \to \HH^1(k,E[\phi]) \to \HH^1(K,E[\phi]) \to \Hom_{G_k}(G,E[\phi]) \to \HH^2(k,E[\phi]) = 0\,,
		\]
		where the final term is $0$ because $k$, being a finite field, has cohomological dimension $1$. By the Weil conjectures $\HH^1(k,E) = 0$. So the connecting homomorphism induces an isomorphism $\partial_\phi: E'(k)/\phi(E(k)) \simeq \HH^1(k,E[\phi])$. By Proposition \ref{prop:selunramified} the $\phi$-Selmer group and its image in $\Hom_{G_k}(G,E[\phi])$ are the unramified subgroups. The maximal unramfied extension of $\kbar K$ of exponent $\ell$ corresponds to multiplication by $\ell$ on $\kbar$, so every unramified (continuous) homomorphism $G \to E[\phi]$ factors through $F[\ell]$. This gives an isomorphism $\Hom_{G_k}(G,E[\phi])^\text{unr} \to \Hom_{G_k}(F[\ell],E[\phi])$.

		It remains to show that the square involving $r_\phi$ commutes. For this suppose $(\alpha : F \to E') \in E'(K)$ is a morphism. Adjusting by $E'(k)$ if necessary we may assume $\alpha \in \Hom(F,E')$ is an isogeny. The fibered product $F' = F \times_{E'} E$ along $\alpha$ and $\phi$ corresponds to an extension $L/ K$ for which there is a point $\beta \in E(L)$ such that $\phi\beta = \alpha$. The cocycle $\eta : \Gal(\Kbar/K) \to E[\phi]$ given by $\sigma \mapsto {}^\sigma(\beta) - \beta$ represents $\partial_\phi(\alpha)$. The action of $\sigma$ is translation by some $P_\sigma \in \ker(\psi)$ so that $({}^\sigma \beta)(x) = \beta(x+P_\sigma)$. Then ${}^\sigma\beta - \beta = \beta(P_\sigma)$. We have a commutative diagram
		\[
			\xymatrix{
				F \ar[r]^{\psi^*} \ar[d]^\alpha & F' \ar[r]^\psi \ar[d]^\beta & F \ar[d]^\alpha \\
				E' \ar[r]^{\phi^*} & E \ar[r]^\phi & E'	
			}
		\]
		where the rows are multiplication by $n = \deg(\phi)$. The cocycle $\eta$ factors through $F[n]$ and from the commutativity of the diagram we see that the map $\sigma \mapsto \beta(P_\sigma)$ is given by $r_\phi(\alpha) : F[n] \to E[\phi]$.
	\end{proof}
	
	When $\phi$ is multiplication by $n$ the theorem together with Remark~\ref{rem:phi=l} gives the following. 
	
	\begin{Corollary}\label{cor:lSelmer}
		Suppose $E,F$ are elliptic curves over the finite field $k$ and $n$ is an integer not divisible by the characteristic of $k$. With $K = k(F)$, there is an exact sequence
		\[
			0 \to \frac{\Hom_{G_k}(F,E)}{n\Hom_{G_k}(F,E)} \to \Hom_{G_k}(F[n],E[n]) \to \Sha(E/K)[n] \to 0\,.
		\]
		If $E$ and $F$ are isogenous ordinary elliptic curves, then $\Sha(E/K)[n]$ has rank at most $2$ over $\Z/n$.
	\end{Corollary}
	
	This says that the nontrivial elements of $\Sha(E/K)[n]$ arise from Galois-module morphisms $F[n] \to E[n]$ that do not come from an isogeny $F \to E$. From this one sees quite directly how Tate's isogeny theorem is equivalent to the finiteness of $\Sha(E/K)[n^\infty]$.

	\begin{Corollary}
	\label{sel-calc}
		Suppose $\phi : E \to E_1$ is an isogeny with kernel generated by a $k$-rational point of order $n$. Then 
		$\Sel^\phi(E/K)$ sits in a split exact sequence
		\[
			0 \to k^\times/k^{\times n} \to \Sel^\phi(E/K) \to F[n](k) \to 0\,.
		\]
		If $E[n] \subset E(k)$, then $\Sel^n(E/K)$ and sits in a split exact sequence
		\[
			0 \to k^\times/k^{\times n}\oplus k^\times/k^{\times n} \to \Sel^n(E/K) \to F[n](k)\oplus F[n](k) \to 0\,.
		\]
	\end{Corollary}	

 \begin{Theorem}
\label{sha-phi}
 	Suppose $\phi: E \to E_1$ is an isogeny of prime degree $\ell$, defined over $k$ whose kernel
	is generated by a $k$-rational point, where $E,E_1$ are ordinary and let
	$F$ be in the same component of the $\ell$-isogeny graph as $E$ and let $K=k(F)$. 
	Then the following are equivalent
 	\begin{enumerate}
 		\item\label{it1} $h_\ell(E_1) < h_\ell(E) \le h_\ell(F)$;
 		\item\label{it2} $\Hom(F,E_1) = \phi \Hom(F,E)$;
 		\item\label{it3}  $E_1(K)/\phi(E(K)) = E_1(k)/\phi(E(k))$;
 		\item\label{it4} $\Sha(E/K)[\phi] = \Sha(E/K)[\ell]$ has rank $2$.
	\end{enumerate}
	
	Denote by $\sigma$ the triple 
	$$(\rk E_1(K)/\phi(E(K)),\rk \Sel^\phi(E/K),\rk \Sha(E/K)[\phi])$$ 
	where $\rk$ is the rank over $\Z/\ell$.
	When the above equivalent conditions hold $\sigma = (1,3,2)$. Otherwise
	$\sigma = (2,2,0)$ if $h_\ell(F)=0$, $\sigma = (3,3,0)$ if $h_\ell(E) < h_\ell(E_1)\le h_\ell(F)$
	and $\sigma = (2,3,1)$, otherwise.
	
\end{Theorem}


\begin{proof}
The equivalence of \eqref{it2} and \eqref{it3} follows immediately from the split exact sequences \eqref{split} for $E$ and $E_1$.

	Suppose $h_\ell(E_1) < h_\ell(E) \le h_\ell(F)$. Let $\alpha \in \Hom(F,E_1)$ be any isogeny. We can factor $\alpha$ as $\alpha_\ell \circ \alpha'$ with $\alpha' : F \to F'$ of degree prime to $\ell$ and $\alpha_\ell:F' \to E_1$ of $\ell$-primary degree. By Theorem~\ref{kohel} $h_\ell(F) = h_\ell(F')$. Thus $\alpha_\ell$ is a path on the $\ell$-volcano consisting of downward edges, so on one the trees (see remark \ref{trees}) and thus must factor through $\phi$.
Since $\alpha_\ell$ has $\ell$-primary degree this implies that $\alpha_\ell$, and hence also $\alpha$, lies in $\phi\Hom(F,E)$.
	
	Let us now consider the cases $h_\ell(E) < h_\ell(E_1),h_\ell(F)$ or $h_\ell(E) > h_\ell(E_1),h_\ell(F)$. In either case there is an isogeny $\alpha : F \to E_1$ with $v_\ell(\deg(\alpha)) = |h_\ell(F)-h_\ell(E_1)|$ which never reaches the level of $E$.
This isogeny can be constructed as follows. Take $F'$ on the tree containing $E_1$ at the level of $F$ which, in addition, is a child of $E_1$ if $h(E_1)<h(F)$.
Then by Theorem \ref{kohel} there is an isogeny $F \to F'$ of degree prime to
$\ell$. There is also a path on the tree from $F'$ to $E_1$ not going through $E$ which defines an isogeny and the isogeny claimed earlier is the composition of those two.
This isogeny cannot possibly lie in $\phi \Hom(F,E)$, since any $\gamma \in \Hom(F,E)$ has $v_\ell(\deg(\gamma)) \ge |h_\ell(F)-h_\ell(E)| > |h_\ell(F)-h_\ell(E_1)|$. Thus \eqref{it2} does not hold.
	
	It remains to consider the case when $h_\ell(E_1) \ge h_\ell(E)$. Lemma \ref{lem1} handles this in the case $F = E$ or is in the branch of the isogeny graph
below $E$. Otherwise, it is clear that there is a path on the isogeny graph
from $F$ to $E_1$ not passing through $E$, showing that \ref{it2},\ref{it3} do not
hold and, from this and the descent sequence \ref{des-seq}, it follows that \ref{it4} does
not hold either.
	
To complete the proof, let $\sigma=(a,b,c)$ and note that $a+c=b$ and that $b=2$ if $h_\ell(F)=0$ and $b=3$, otherwise. Also, $a \ge 1$ because of the contribution of
torsion. In order to have $a=1$ we must have that all isogenies $F \to E_1$ factor through $E$. This can only be the
case when $h_\ell(F) \ge h_\ell(E) > h_\ell(E_1)$. From this we can already conclude that $\sigma = (2,2,0)$ if $h_\ell(F)=0$. We also
get that $\sigma = (1,3,2)$ if $h_\ell(F) \ge h_\ell(E) > h_\ell(E_1)$.

We claim that $a=3$ precisely when $h_\ell(F) \ge h_\ell(E_1) > h_\ell(E)$. Just as in the previous argument, an isogeny $F \to E$
must factor through $E_1$ via $\hat\phi$ and so an isogeny in $\phi(E(K))$ must factor through $\phi \circ \hat\phi = [\ell]$
and $E_1(K)/\phi(E(K)) = E_1(K)/\ell(E_1(K))$ which has rank $3$. Hence $\sigma = (3,3,0)$.

If $h_\ell(E_1)=h_\ell(E)$, then $a=2$, so $\sigma = (2,3,1)$ unless $h_\ell(F)=0$. If $0 < h_\ell(F) \le h_\ell(E) < h_\ell(E_1)$, then
$\sigma = (2,3,1)$ as we can find a path from $E_1$ to some $E'$ at the level of $F$ not passing through $E$ and 
we can then use an isogeny $F \to E'$ of degree prime to $\ell$. Similarly, for $0 < h_\ell(F) \le h_\ell(E_1) < h_\ell(E)$, 
$\sigma = (2,3,1)$ also. These are all the possibilities since $|h_\ell(E) - h_\ell(E_1)| \le 1$.

\end{proof}	

\begin{Remark}
If there is an isogeny $\phi: E \to E_1$ of degree $\ell$ whose kernel is not generated by a $k$-rational point, then the
$\ell$-isogeny graph has height $m=0$. In this case, one verifies that $\sigma = (1,1,0)$.
\end{Remark}

\section{Calculation of $\Sha$}

\begin{Theorem}\label{thm:sha}
	Suppose that $E$, $F$ are isogenous ordinary elliptic curves over $k$ and set $\mathcal{O} = \End(E) \cap \End(F)$, the intersection taking place in $\Q(\pi)$. There is an isomorphism of groups $\Sha(E/k(F)) \simeq \mathcal{O}/\Z[\pi] \times \mathcal{O}/\Z[\pi]$. In particular, for any prime $\ell$ we have $v_\ell(\#\Sha(E/k(F))) = 2 \min \{ h_\ell(E),h_\ell(F) \}$.
\end{Theorem}	 

\begin{Remark}
	This generalizes the observation of Milne~\cite[p 102]{Milne} that $\#\Sha(E/k(E)) = [\End(E):\Z[\pi]]^2$.
\end{Remark}

\begin{Remark}
	It would be interesting to find a meaningful interpretation of the isomorphism of the theorem and, In particular, identify a natural symplectic pairing on $\mathcal{O}/\Z[\pi] \times \mathcal{O}/\Z[\pi]$ which corresponds to the Cassels-Tate pairing.
\end{Remark}

\begin{proof}
	
	Let $\ell$ be a prime. Since $\Hom_{G_k}(F,E)/\ell$ has dimension $2$ over $\F_\ell$, Corollary~\ref{cor:lSelmer} shows that $\Sha(E/k(F))[\ell]$ has dimension $0$ or $2$. Therefore $\Sha(E/k(F))[\ell^\infty]$ is a product of two cyclic groups. Since $\mathcal{O}/\Z[\pi]$ is cyclic and $v_\ell([\mathcal{O}:\Z[\pi]]) = \min\{h_\ell(E),h_\ell(F)\}$ we see that the first statement of the theorem follows from the claim that, for each $\ell$, $v_\ell(\#\Sha(E/k(F))[\ell^\infty]) = 2 \min\{ h_\ell(E),h_\ell(F)\}$. This holds trivially for the $p$-part as $\Sha(E/k(F))[p] = 0$ and the $p$-isogeny graph has just one level (see Section~\ref{sec:p}).
	
	Note that $\Sha(E/k(F)) \simeq \Sha(F/k(E))$, since both are isomorphic to the Brauer group of the surface $E\times F$ \cite[Theorem 3.1]{Tate-BSD}. By swapping the roles of $E$ and $F$ if necessary we may therefore assume that $h_\ell(E) \le h_\ell(F)$. Let $E'$ be the unique elliptic curve lying below $F$ in the $\ell$-isogeny graph (i.e., such that there is a downward isogeny $F \to E'$, possibly the identity) with $h_\ell(E') = h_\ell(E)$. Since $h_\ell(E) = h_\ell(E')$, Theorem~\ref{kohel} implies that there is an isogeny $\alpha : E' \to E$ of degree prime to $\ell$. This isogeny induces an isomorphism $\Sha(E'/k(F)[\ell^\infty]\simeq \Sha(E/k(F))[\ell^\infty]$. So, replacing $E$ by $E'$ if needed, we are reduced to the case that $E$ lies below $F$ in the $\ell$-isogeny graph.
		
	Suppose the characteristic polynomial of Frobenius on $F/k$ factors as $x^2 + tx + q = (x-a)(x-b)$. By \cite[Theorem 3]{Milne} we have 
		\begin{equation}\label{eq:BSD}
			\#\Sha(E/k(F)) \cdot R_{E/k(F)} = q\left(1-\frac{a}{b}\right)\left(1-\frac{b}{a}\right)\,,
		\end{equation}
		where $R_{E/k(F)} = \left| \det \Tr(\alpha_i\beta_j) \right|,$ for any bases $\{ \alpha_i \}$ and $\{ \beta_i\}$ of the free $\Z$-modules $\Hom(F,E)$ and $\Hom(E,F)$. By Lemma~\ref{lem1} we have $\Hom(F,E) = \phi \End(F)$. It follows that $R_{E/k(F)} = \deg(\phi)^2 R_{F/k(F)}$. Comparing~\eqref{eq:BSD} for $F/k(F)$ and $E/k(F)$ shows that $\#\Sha(F/k(F)) =\deg(\phi)^2 \#\Sha(E/k(F)) = \ell^{2(h_\ell(F)-h_\ell(E))} \#\Sha(E/k(F))$. It thus suffices to verify the claim when $F = E$.
		
		Let $\frak{f},\frak{g} \in \Z$ be the conductors of the orders $\Z[\pi]$ and $\End(F)$ in the quadratic imaginary field $L = \End(F) \otimes \Q$, so that $\Z[\pi] = \Z[\frak{f}D]$, $\End(F) = \Z[\frak{g}D]$ and $\mathcal{O}_L = \Z[D]$, where $D = \frac{d_L + \sqrt{d_L}}{2}$, $d_L$ is the fundamental discriminant. In particular $\End(F)/\Z[\pi] \simeq \Z/n\Z$ where $n = \frak{f}/\frak{g}$. Then
	\[
		R_{F/k(F)} = \left| \det
	\begin{bmatrix}
		\Tr([1]) & \Tr(\frak{g}D) \\ \Tr (\frak{g}D) & \Tr (\frak{g}^2D^2) \,
	\end{bmatrix} \right|\,.
	\]
	For any integer $h$ we have 
	\[  \left| \det
	\begin{bmatrix}
		\Tr([1]) & \Tr(hD) \\ \Tr (hD) & \Tr (h^2D^2) 
	\end{bmatrix}
	\right| = \left|	\begin{bmatrix}
		2 & hd_L \\ hd_L & h^2(d_L^2 + d_L)/2 ) 
	\end{bmatrix}
	\right| = h^2d_L
	\]
	As $\pi = \frak{f}D$ we have
	\[
	\frak{f}^2d_L = \left|\det
	\begin{bmatrix}
		 \Tr([1]) & \Tr(\pi) \\ \Tr (\pi) & \Tr (\pi^2) 
	\end{bmatrix}\right|
	= |2(a^2 + b^2) - (a+b)^2| = (a-b)(b-a) = q\left(1-\frac{b}{a}\right)\left(1-\frac{a}{b}\right) \]
	Comparing with~\eqref{eq:BSD} we see that $\#\Sha(F/k(F)) = [\End(F):\Z[\pi]]^2$, as required.
\end{proof}

\begin{Corollary}
	$E$ has height $m$ in the $\ell$-isogeny graph if and only if $\Sha(E/k(E))[\ell^\infty] \simeq \Z/\ell^m\times \Z/\ell^m$.
\end{Corollary}

\begin{Corollary}
(\cite[Theorem 5]{MMSTV}) If $E$ has full rational $\ell^m$-torsion, then $E$ has height at least $m$ in the $\ell$-isogeny graph.
\end{Corollary}	

\begin{proof}
	If $E(k)$ contains the $\ell^m$-torsion, then by Corollary~\ref{cor:lSelmer} we have that $\Sha(E/k(E))[\ell^m] \simeq \Z/\ell^m\Z \times \Z/\ell^m\Z$. So the result follows from the previous corollary.
\end{proof}

\begin{Corollary}
	Any downward isogeny of degree $\ell^n$ must annihilate the $\ell^n$-torsion in $\Sha(E/k(E))$.
\end{Corollary}

\begin{proof}
	Suppose $E$ is at level $m > 0$ in the $\ell$-isogeny graph. Then $\Sha(E/k(E))[\ell^\infty] \simeq \Z/\ell^m\Z \times \Z/\ell^m\Z$. By Theorem~\ref{sha-phi} any composition of $m$ downward $\ell$-isogenies must kill $\Sha(E/k(E))[\ell^\infty]$, but each step can only kill $\ell$-torsion. Hence each step must kill all of the $\ell$-torsion.
\end{proof}

\section{Concluding remarks}\label{sec:p}

The questions addressed in this paper can be considered also for supersingular curves and/or in the case $\ell = p$
although (apart from the supersingular $\ell = p$ case) the situation is much simpler.

If $E$ is ordinary and $\ell = p$, there are only 
two degree $p$ isogenies from $E$, the Frobenius and Verschiebung. The
component of the isogeny graph containing $E$ is thus a cycle consisting 
of the conjugates of $E$. The Selmer group and $E_1(K)/\phi(E(K))$ are
one dimensional if $\phi$ is Frobenius or if $\phi$ is Verschiebung and
the $p$-torsion of $E$ is not defined over $K$. When $\phi$ is Verschiebung and
the $p$-torsion of $E$ is defined over $K$, both these groups are two dimensional.
The $p$-torsion of $\Sha(E/K)$ is trivial in both cases. See \cite[Proposition 3.3]{Ulmer}, but
note that the Selmer group in the case $A=1$ of the Verschiebung descent has dimension one more
than reported there.

Now consider the case when $E$, $F$ are isogenous supersingular curves over $k$. If $\F_{p^2} \subset k$, then $\Hom_{G_k}(F,E)$ has rank $4$ and so Corollary~\ref{cor:lSelmer} shows that $\Sha(E/k(F))$ is a $p$-group. If $\F_{p^2} \not\subset k$, then the connected components of the $\ell$-isogeny graph of $E$ are volcanoes with at most two levels (for $\ell = 2$, \cite{DG}). Moreover, if there are two levels, then $\ell = 2$. Indeed, if $E$ is supersingular and $E[\ell] \subset E(k)$, then $\ell | (q-1)$. If, in addition, $\F_{p^2} \not\subset k$ then $\# E(k) = q+1$ and we get that $\ell | (q+1)$, so $\ell = 2$.
The analogues of Theorems \ref{sha-phi} and \ref{thm:sha} hold when $\F_{p^2} \not\subset k$, with $\End(E)$ an order of the quadratic field $\Q(\sqrt{-p})$. 

Finally, we remark that, as noted in \cite{Milne}, if $E$, $F$ are non-isogenous elliptic curves over $k$,
then it follows from the Birch-Swinnerton-Dyer formula that 
$$\# \Sha(E/k(F)) = (\#E(k) - \#F(k))^2.$$

\section*{Acknowledgements}
The authors were supported by the Marsden Fund Council administered by the Royal Society of New Zealand. They would like to thank Doug Ulmer for correspondence regarding \cite{Ulmer}.



\end{document}